\documentclass[a4paper]{article}
\usepackage{geometry}
\usepackage{amssymb}
\usepackage{latexsym}
\usepackage{amsmath}
\usepackage{indentfirst}
\usepackage{graphicx}
 \usepackage[colorlinks=true]{hyperref}
\usepackage{mathrsfs} 
\usepackage{chngcntr}
\usepackage{color}

\newtheorem{Thm}{Theorem}[section]

\newtheorem{Lem}{Lemma}[section]

\numberwithin{equation}{section}

\newenvironment{proof}{\medskip\par\noindent{\bf Proof\/}:\quad}{\qquad
\raisebox{-0.5mm}{\rule{1.5mm}{1mm}}\vspace{6pt}}
\begin{document}
\title{Critical fractional Kirchhoff problems: Uniqueness and Nondegeneracy}

\author{
{Zhipeng Yang$^{1}$}\thanks{Corresponding author:yangzhipeng326@163.com} and {Yuanyang Yu$^{2}$}\thanks{yyysx43@163.com}\\
\small Department\, of \,Mathematics, Yunnan\, Normal\, University, Kunming, 650500, P.R.China.$^{1}$\\
\small School of Mathematics and Statistics, Yunnan University, Kunming 650500, People’s Republic of China$^{2}$\\
}
\date{} \maketitle
\textbf{Abstract.}
In this paper, we consider the following critical fractional Kirchhoff equation
\begin{equation*}
\Big(a+b{\int_{\mathbb{R}^{N}}}|(-\Delta)^{\frac{s}{2}}u|^2dx\Big)(-\Delta)^su=|u|^{2^*_s-2}u,\quad \text{in}\ \mathbb{R}^{N},
\end{equation*}
where $a,b>0$, $\frac{N}{4}<s<1$, $2^*_s=\frac{2N}{N-2s}$ and $(-\Delta )^s$ is the fractional Laplacian. We prove the uniqueness and nondegeneracy of positive solutions to the problem, which can be used to study the singular perturbation problems concerning fractional Kirchhoff equations.

\vspace{6mm} \noindent{\bf Keywords:} Fractional Kirchhoff equations; Uniqueness; Nondegeneracy.

\vspace{6mm} \noindent


\section{Introduction and main results}

In this paper, we are concerned with the following fractional Kirchhoff problem
\begin{equation}\label{eq1.2}
\Big(a+b{\int_{\mathbb{R}^{N}}}|(-\Delta)^{\frac{s}{2}}u|^2dx\Big)(-\Delta)^su=u^{2^*_s-1},\quad \text{in}\ \mathbb{R}^{N},
\end{equation}
where $a,b>0$ are parameters, $0<\frac{s}{2}<\frac{N}{4}<s<1$, $(-\Delta )^s$ is the pseudo-differential operator defined by
\begin{equation*}
\mathcal{F}((-\Delta)^su)(\xi)=|\xi|^{2s}\mathcal{F}(u)(\xi),\ \ \xi\in \mathbb{R}^N,
\end{equation*}
where $\mathcal{F}$ denotes the Fourier transform, and $2^*_s=\frac{2N}{N-2s}$ is the standard fractional Sobolev critical exponent.
\par
If $s=1$, equation \eqref{eq1.2} reduces to the well known Kirchhoff type problem, which and their variants have been studied extensively in the literature. The equation that goes under the name of Kirchhoff equation was proposed in \cite{Kirchhoff1883} as a model for the transverse oscillation of a stretched string in the form
\begin{equation}\label{eq1.3}
\rho h \partial_{t t}^{2} u-\left(p_{0}+\frac{\mathcal{E}h}{2 L} \int_{0}^{L}\left|\partial_{x} u\right|^{2} d x\right) \partial_{x x}^{2} u=0,
\end{equation}
for $t \geq 0$ and $0<x<L$, where $u=u(t, x)$ is the lateral displacement at time $t$ and at position $x, \mathcal{E}$ is the Young modulus, $\rho$ is the mass density, $h$ is the cross section area, $L$ the length of the string, $p_{0}$ is the initial stress tension.
Problem \eqref{eq1.3} and its variants have been studied extensively in the literature. Bernstein obtains the global stability result in \cite{Bernstein1940BASUS}, which has been generalized to arbitrary dimension $N\geq 1$ by Poho\v{z}aev in \cite{Pohozaev1975MS}. We also point out that such problems may describe a process of some biological systems dependent on the average of itself, such as the density of population (see e.g. \cite{Arosio-Panizzi1996TAMS}).  From a mathematical point of view, the interest of studying Kirchhoff equations comes from the nonlocality of Kirchhoff type equations. We refer to \cite{MR3987384} for a recent survey of the results connected to this model.
\par
On the other hand, the interest in generalizing to the fractional case the model introduced by Kirchhoff does
not arise only for mathematical purposes. In fact, following the ideas of \cite{MR2675483} and the concept
of fractional perimeter, Fiscella and Valdinoci proposed in \cite{MR3120682} an equation describing the
behaviour of a string constrained at the extrema in which appears the fractional length of the rope. Recently, problem similar to \eqref{eq1.2} has been extensively investigated by many authors using different techniques and producing several relevant results  (see, e.g. \cite{MR4056169,Valdinocibook2017,MR4245633,MR3985380,Gu-Yang,MR4021897}).
\par
In this paper, we study the uniqueness and nondegeneracy of positive solutions to problem \eqref{eq1.2}. These quantitative properties play a fundamental role in singular perturbation problems.  For the fractional Schr\"{o}dinger equation such as
\begin{equation}\label{eq1.5}
(-\Delta)^su+u=u^p, \quad \text{in} \,\,\mathbb{R}^{N}.
\end{equation}
In contrast to the classical limiting case when $s=1$, in which standard ODE techniques are applicable, uniqueness of ground state solutions to Eq. \eqref{eq1.5} is a really difficult problem. In the case that $s=\frac{1}{2}$ and $N=1$, Amick and Toland \cite{MR1111746}, they obtained the uniqueness result for solitary waves of the Benjamin-Ono equation. After that, Lenzmann \cite{MR2561169} obtained the uniqueness of ground states for the pseudorelativistic Hartree equation in $3$-dimension. In \cite{MR3070568}, Frankand and Lenzmann extends the results in \cite{MR1111746} to the case that $s \in(0,1)$ and $N=1$ with completely new methods. For the high dimensional case, Fall and Valdinoci \cite{MR3207007} established the uniqueness and nondegeneracy of ground state solutions of \eqref{eq1.5} when $s \in(0,1)$ is sufficiently close to $1$ and $p$ is subcritical. In their striking paper \cite{MR3530361}, Frank, Lenzmann and Silvestre solved the problem completely, and they showed that the ground state solutions of \eqref{eq1.5} is unique for arbitrary space dimensions $N \geq 1$ and all admissible and subcritical exponents $p>0 .$ For a systematical research on applications of nondegeneracy of ground states to perturbation problems, we refer to Ambrosetti and Malchiodi \cite{Ambrosetti-Malchiodibook} and the references therein.
\par
One of the main idea is based on the scaling technique which allows us to find a relation between solutions of  \eqref{eq1.2} and \eqref{eq1.5}. In order to state our main results, we recall some preliminary results for the fractional Laplacian. For $0<s<1$,
denote by $D=D^{s,2}(\mathbb{R}^N)$ the completion of $C^{\infty}_0(\mathbb{R}^N))$ under the seminorm
\begin{equation*}
  \|u\|^2_D=\int_{\mathbb{R}^N}|(-\Delta)^{\frac{s}{2}}u|^2dx.
\end{equation*}
From \cite{Nezza-Palatucci-Valdinoci2012BSM}, we have
\begin{equation*}
\|(-\Delta)^{\frac{s}{2}}u\|^2_2=\int_{\mathbb{R}^N}|\xi|^{2s}|\mathcal{F}(u)|^2d\xi=\frac{1}{2}C(N,s)
\int_{\mathbb{R}^N\times\mathbb{R}^N}\frac{|u(x)-u(y)|^2}{|x-y|^{N+2s}}dxdy
\end{equation*}
and the fractional Gagliardo-Nirenberg-Sobolve inequality
\begin{equation}\label{eq1.8}
\int_{\mathbb{R}^N}|u|^{p+1}dx\leq\mathcal{S}\Big(\int_{\mathbb{R}^N}|(-\Delta
)^{\frac{s}{2}}u|^2dx\Big)^{\frac{N(p-1)}{4s}}
\Big(\int_{\mathbb{R}^N}|u|^2dx\Big)^{\frac{p-1}{4s}(2s-N)+1},
\end{equation}
where $\mathcal{S}>0$ is the best constant. It follows from
\eqref{eq1.8} that
\begin{equation}\label{eq1.9}
J(u)=\frac{\mathcal{S}\Big(\int_{\mathbb{R}^N}|(-\Delta)^{\frac{s}{2}}u|^2dx\Big)^{\frac{N(p-1)}{4s}}\Big(\int_{\mathbb{R}^N}|u|^2dx\Big)^{\frac{p-1}{4s}(2s-N)+1}}
{\int_{\mathbb{R}^N}|u|^{p+1}dx}>0.
\end{equation}
It has been pointed out that in \cite{MR3530361} that $\mathcal{S}=\inf\limits_{u\in H^s(\mathbb{R}^N)\backslash\{0\}}J(u)$ is actually
attained. Here we say that $Q$ is a ground state solution if a solution of \eqref{eq1.5} is also a minimizer for $J(u)$.
\par
For fractional Kirchhoff problems, not much is known in this respect. Recently, R\v{a}dulescu and Yang \cite{R-Yang} established uniqueness and nondegeneracy for positive solutions to Kirchhoff equations with subcritical growth. More precisely,
they proved that the following fractional Kirchhoff equation
\begin{equation*}
\Big(a+b{\int_{\mathbb{R}^{N}}}|(-\Delta)^{\frac{s}{2}}u|^2dx\Big)(-\Delta)^su+mu=|u|^{p-2}u,\quad \text{in}\ \mathbb{R}^{N},
\end{equation*}
where $a,b,m>0$, $\frac{N}{4}<s<1$, $2<p<2^*_s=\frac{2N}{N-2s}$, has a unique nondegenerate positive radial solution. As a counterpart to this result, we have the following theorem for fractional Kirchhoff equations with critical Sobolev growth.

\begin{Thm}\label{Thm1.1}Let $a,b>0$ and $\frac{N}{4}<s<1$. Then equation \eqref{eq1.2} has a unique ground state solution $U\in D$  up to translation. Moreover, $U$ is nondegenerate in the sense that there holds
$$\ker \mathcal{L}_{+} =span\{\partial_{x_1}U, \partial_{x_2}U,\partial_{x_3}U,\frac{N-2s}{2}U+x\cdot\nabla u\},$$
where $\mathcal{L}_{+} $ is defined as
\begin{equation}\label{eq1.6}
\mathcal{L}_{+} \varphi=\Big(a+b{\int_{\mathbb{R}^{N}}}|(-\Delta
)^{\frac{s}{2}}U|^2dx\Big)(-\Delta
)^s\varphi-(2^*_s-1)U^{2^*_s-2}\varphi+2b\Big({\int_{\mathbb{R}^{N}}}(-\Delta
)^{\frac{s}{2}}U(-\Delta )^{\frac{s}{2}}\varphi dx\Big)(-\Delta )^sU
\end{equation}
acting on $L^2(\mathbb{R}^N)$ with domain $D$.
\end{Thm}

\section{Proof of Theorem \ref{Thm1.1}}
In this section we prove Theorem \ref{Thm1.1}. Our methods depend on the following result for the well-known fractional critical problem
\begin{equation}\label{eq5.1}
(-\Delta)^{s} u=u^{2_{s}^{*}-1}, \quad x \in \mathbb{R}^{N}.
\end{equation}
By using the moving planes method of integral form, Chen, Li and Ou \cite{MR2200258} proved that every positive regular solution of \eqref{eq5.1} is radially symmetric and monotone about some point, and therefore assumes the form
\begin{equation*}
Q(x)=C(N, s)\left(\frac{\mu}{\mu^{2}+|x-\xi|^{2}}\right)^{\frac{N-2s}{2}}, \quad C(N, s) \neq 0, \mu>0, \quad \xi \in \mathbb{R}^{N}.
\end{equation*}
\subsection{Proof of Uniqueness}
Let $Q$ be the uniquely positive solution of \eqref{eq1.5} and also a minimizer of $J(u)$. Consider the
equation
\begin{equation}\label{eq2.1}
f(\mathcal{E})=\mathcal{E}-a-b\|(-\Delta)^{\frac{s}{2}}Q\|^2_2\mathcal{E}^{\frac{N-2s}{2s}}=0,
\quad \mathcal{E}\in (a,+\infty).
\end{equation}
Recall that $s>\frac{N}{4}$, we have $\frac{N-2s}{2s}<1$, which implies that $\lim\limits_{\mathcal{E}\to +\infty}f(\mathcal{E})=+\infty$.
Moreover, one has $f(a)<0$. Consequently, there exists $\mathcal{E}_0>a$ such that $f(\mathcal{E}_0)=0$. Let
\begin{equation}\label{eq2.2}
U(x)=Q(\mathcal{E}_0^{-\frac{1}{2s}}x)=Q(\tilde{x}).
\end{equation}
It follows from the definition of $(-\Delta)^s$ that
\begin{equation*}
\Big((-\Delta)^sU\Big)(x)=\mathcal{E}_0^{-1}\Big((-\Delta)^sQ\Big)(\tilde{x}),
\end{equation*}
which implies that $U$ is a positive solution of the following equation
\begin{equation}\label{eq2.3}
\mathcal{E}_0(-\Delta )^sU=U^{2^*_s-1} \quad \text{in}\quad \mathbb{R}^{N}.
\end{equation}
Note that
\begin{equation*}
\|(-\Delta)^{\frac{s}{2}}U\|^2_2=\|(-\Delta)^{\frac{s}{2}}Q\|^2_2\mathcal{E}_0^{\frac{N-2s}{2s}}.
\end{equation*}
We have $\mathcal{E}_0=a+b\|(-\Delta)^{\frac{s}{2}}U\|^2_2$. From \eqref{eq2.3}, we conclude that $U$ is a positive solution of
\eqref{eq1.2}. Furthermore, $U$ is a minimizer for $J(u)$ since $J(U)=J(Q)$.  Let $U$ be a ground state positive solution of \eqref{eq1.2} and set
\begin{equation*}
\mathcal{E}_0=a+b\|(-\Delta)^{\frac{s}{2}}U\|^2_2\quad \text{and}\quad \tilde{U}(x)=U(\mathcal{E}_0^{\frac{1}{2s}}x).
\end{equation*}
It is easy to check that $\tilde{U}$ is a positive solution of \eqref{eq1.5} and a minimizer of
$J(u)$. Hence, $\tilde{U}$ is a ground state solution. Then there exists some $x_0\in \mathbb{R}^N$ such that $\tilde{U}(x)=Q(x-x_0)$, where $Q$ is the unique ground state of \eqref{eq1.5}. Consequently, we have
\begin{equation*}
{U}(x)=Q\Big(\mathcal{E}_0^{-\frac{1}{2s}}x-x_0\Big).
\end{equation*}
and
\begin{equation}\label{eq3.1}
\mathcal{E}_0=a+b\|(-\Delta)^{\frac{s}{2}}Q\|^2_2\mathcal{E}_0^{\frac{N-2s}{2s}}.
\end{equation}
\par
 In order to arrive at the desired conclusion, we have to prove  $\mathcal{E}_0$ is uniquely defined. For this, consider the real function $f$
 defined by \eqref{eq2.1}. Recall that $\frac{N-2s}{2s}<1$. It is evident that $f$ is strictly  increasing on $(a,+\infty)$ if $\frac{N-2s}{2s}\leq 0$. Moreover,  for $0<\frac{N-2s}{2s}<1$, one can verify that the solution  $\mathcal{E}_0$ of \eqref{eq2.1} must belong to the increasing
 interval  of $f$ since a solution $\mathcal{E}$ satisfies
\begin{equation*}
\mathcal{E}-b\|(-\Delta)^{\frac{s}{2}}Q\|^2_2\mathcal{E}^{\frac{N-2s}{2s}}>0.
\end{equation*}
Hence, $\mathcal{E}_0$ is uniquely defined by \eqref{eq3.1}. The proof is completed.

\subsection{Proof of nondegeneracy}
Differentiating the equation
$$
(-\Delta)^{s} Q=Q^{2^*_s-1} \quad \text { in } \mathbb{R}^{N},
$$
with respect of the parameters at $\mu=1, \xi=0$, we see that the functions
$$
\partial_{\mu} Q=\frac{N-2 s}{2}Q+x \cdot \nabla Q, \quad \partial_{\xi_{i}} Q=-\partial_{x_{i}}Q
$$
annihilate the linearized operator around $Q$; namely, they satisfy the equation
$$
(-\Delta)^{s} \phi=(2^*_s-1)Q^{2^*_s-2} \phi \quad \text { in } \mathbb{R}^{N}.
$$
With no loss of generality, we assume that $U(x)=U(|x|)$ is the unique positive radial energy solution to Eq. \eqref{eq1.2}. Then we have
\begin{equation}\label{6}
\operatorname{Ker}\mathcal{L}_+=\operatorname{span}\left\{\frac{N-2s}{2}U+x \cdot \nabla U, U_{x_{1}}, U_{x_{2}}, \cdots, U_{x_{N}}\right\}.
\end{equation}
Note that $\frac{N-2s}{2}U+x \cdot \nabla U=\frac{N-2s}{2}U+r U^{\prime}(r)$ with $r=|x|$ is a radial function in $D$. For simplicity, denote $D_{\mathrm{rad}}=\{v \in D: v(x)=v(|x|)\}$.

\begin{Lem}\label{Lem3.1}
Let $\mathcal{L}_{+} \varphi=0$ for $\varphi \in D_{\mathrm{rad}}$, then $\varphi=\lambda(\frac{N-2s}{2}U+x \cdot \nabla U)$ for some $\lambda \in \mathbb{R}$.
\end{Lem}
\begin{proof}Direct computation shows that $\frac{N-2s}{2}U+x \cdot \nabla U=\frac{N-2s}{2}U+r U^{\prime}(r)$ is indeed a radial solution to equation $\mathcal{L}_{+} \varphi=0$. We have to prove that $\frac{N-2s}{2}U+x \cdot \nabla U$ is the unique radial solution to equation $\mathcal{L}_{+} \varphi=0$ in $D_{\mathrm{rad}}$ up to a constant.
\par
Let $c=a+b\|(-\Delta)^{\frac{s}{2}}U\|^2_2$. Recall that $U$ is a ground state solution of \eqref{eq1.2}. It follows from above that $c$ is a constant independent of $U$ under the assumptions of Theorem \ref{Thm1.1}.
Let $\varphi \in D_{\text {rad }}$ satisfy $\mathcal{L}_{+} \varphi=0$. It is equivalent to
$$
\mathcal{A} \varphi\doteq c(-\Delta)^s\varphi-(2^*_s-1)U^{2^*_s-1}\varphi=-2b\Big({\int_{\mathbb{R}^{N}}}(-\Delta)^{\frac{s}{2}}U(-\Delta )^{\frac{s}{2}}v dx\Big)(-\Delta)^sU
$$
Write $e_{0}=\frac{N-2s}{2}U+x \cdot \nabla U$ for simplicity. Since $D_{\text {rad }}$ is a Hilbert space, denote by $D_{0}$ the orthogonal complement of $\mathbb{R} e_{0}$ in $D_{\text {rad }}$. Then $\varphi=\lambda e_{0}+v$ for some $\lambda \in \mathbb{R}$ and $v \in D_{0}$. By a direct computation, we find that $\int(-\Delta)^{\frac{s}{2}}u \cdot (-\Delta)^{\frac{s}{2}} e_{0}=0$. This implies $u \in D_{0}$. Moreover, note that \eqref{6} implies that $\mathcal{A}$ is invertible on $D_{0}$. It follows from $\mathcal{A} e_{0}=0$ and $\int(-\Delta)^{\frac{s}{2}}u \cdot (-\Delta)^{\frac{s}{2}} e_{0}=0$ that $v$ satisfies
\begin{equation}\label{eq3.5}
\mathcal{A}v=-2b\Big({\int_{\mathbb{R}^{N}}}(-\Delta)^{\frac{s}{2}}U(-\Delta )^{\frac{s}{2}}v dx\Big)(-\Delta)^sU=-\frac{2b\sigma_v}{c}(U^{2^*_s-1}),
\end{equation}
where
\begin{equation*}
\sigma_v={\int_{\mathbb{R}^{N}}}(-\Delta )^{\frac{s}{2}}U(-\Delta
)^{\frac{s}{2}}v dx.
\end{equation*}
By applying \cite[Theorem 3.3]{MR3530361}, we conclude that
\begin{equation}\label{eq3.6}
v=-\frac{2b\sigma_v}{c}T_+^{-1}(U^{2^*_s-1})=-\frac{b\sigma_v}{sc}\psi,
\end{equation}
where $\psi=x\cdot \nabla U$. Multiplying \eqref{eq3.6} by $(-\Delta)^sU$ and integrating over $\mathbb{R}^N$, we see that
\begin{equation}\label{eq3.7}
\int_{\mathbb{R}^{N}}v(-\Delta)^{s}Udx=-\frac{b\sigma_v}{sc}\int_{\mathbb{R}^{N}}\psi(-\Delta)^{s}Udx.
\end{equation}
Note that
\begin{equation}\label{eq3.8}
\int_{\mathbb{R}^{N}}v(-\Delta)^{s}Udx=\int_{\mathbb{R}^{N}}(-\Delta )^{\frac{s}{2}}U(-\Delta)^{\frac{s}{2}}v dx
\end{equation}
and
\begin{equation}\label{eq3.9}
\int_{\mathbb{R}^{N}}\psi(-\Delta)^{s}Udx=\frac{2s-N}{2}\int_{\mathbb{R}^{N}}|(-\Delta)^{\frac{s}{2}}U|^2 dx
\end{equation}
(see e.g. \cite{Ros2014The}). We then conclude from \eqref{eq3.7}-\eqref{eq3.9} that
\begin{equation*}
\sigma_v=-\frac{b(2s-N)\sigma_v}{2sc}\int_{\mathbb{R}^{N}}|(-\Delta
)^{\frac{s}{2}}U|^2 dx=-\frac{(c-a)(2s-N)}{2sc}\sigma_v,
\end{equation*}
which implies that $\sigma_v=0$ due to $-\frac{(c-a)(2s-N)}{2sc}\neq1$. From \eqref{eq3.5}, we have $v=0$.
Therefore, $\mathcal{A} v=0$. Recall $v \in D_{0}$. Applying \eqref{6} gives $v=0$. Thus, we obtain $\varphi=\lambda e_{0}$. The proof is complete.
\end{proof}
\par
Since $(-\Delta )^sU=c^{-1}(+U^{2^*_s-2})$ and $U(x)=U(|x|)$ is radial function, the operator $L_+$ commutes with rotations in $\mathbb{R}^N$ (see e.g. \cite{MR2561169}). Therefore, we can decompose $L^2(\mathbb{R}^N)$ using spherical harmonics
$L^2(\mathbb{R}^N)=\oplus_{l\geq 0}\mathcal{H}_l$
so that $L_+$ acts invariantly on each subspace
$\mathcal{H}_l=L^2(\mathbb{R}_+, r^{N-1}dr)\otimes\mathcal{Y}_l.$
Here $\mathcal{Y}_l=span\{Y_{l,m}\}_{m\in M_l}$ denotes space of the spherical harmonics of degree $l$ in space dimension $N$ and $M_l$ is an index set depending on $l$ and $N$.
We can describe the action of $L_+$ more precisely. For each $l$, the action of $L_+$ on the radial factor in $\mathcal{H}_l$ is given
by
\begin{equation*}
(L_{+,l}f)(r)=c\Big((-\Delta_l)^sf\Big)(r)+f(r)-pU^{p-1}(r)f(r)+2b(W_lf)(r)
\end{equation*}
with the nonlocal linear operator
\begin{equation*}
(W_lf)(r)=\frac{2\pi^{\frac{N}{2}}}{\Gamma(\frac{N}{2})}\Big((-\Delta_l)^sU\Big)(r)\int_0^{+\infty}\Big((-\Delta_l)^sU\Big)(r)r^{N-1}f(r)dr
\end{equation*}
for $f\in C_0^{\infty}(\mathbb{R}^+)\subset L^2(\mathbb{R}_+, r^{N-1}dr)$. Here
$(-\Delta_l)^s$ is given by spectral calculus and the known formula
\begin{equation*}
-\Delta_l=-\partial_r^2-\frac{N-1}{r}\partial_r+\frac{l(l+N-2)}{r^2}.
\end{equation*}
\par
Applying  arguments similar to that used in \cite{MR3530361} and \cite{MR2561169}, one can verify that each $L_{+,l}$ enjoys a Perron-Frobenius property, that is, if $E=\inf \sigma(\mathcal{H}_l)$ ia an eigenvalue, then $E$ is simple and the corresponding eigenfunction can be chosen strictly positive. Moreover, we have $L_{+,l}>0$ for $l\geq 2$ in the sense of quadratic forms (see e.g. \cite{MR2561169}).
\par
{\bf Proof of nondegeneracy: } Since $\partial _{x_i}U(x)=U'(r)\frac{x_i}{r}\in \mathcal{H}_1$, this shows that $L_{+,1}U'=0$. Note that $U'(r)<0$. It follows from the Perron-Frobenius property that $0$ is the lowest eigenvalue  of $L_{+,1}$, with $-U'(r)$ being its an corresponding eigenfunction. Therefore, for any $v\in \mathcal{H}_1$ satisfying $L_{+,1}v=0$ must be some linear combination of $\{\partial _{x_i}U:i=1,2,\cdots, N\}$ and $\frac{N-2s}{2}U+x \cdot \nabla U$. Recall that $L_{+,l}>0$ for $l\geq 2$. Applying the Perron-Frobenius property again, $0$ cannot be an eigenvalue of $L_{+,l}>0$ for $l\geq 2$. Finally, Lemma \ref{Lem3.1} implies that $L_{+,0}=\{0\}$. Consequently, for any $v\in \ker L_+$, we conclude that $v\in \mathcal{H}_1$ and hence $\ker L_+=\ker L_{+,1}=span\{\partial_{x_1}U, \partial_{x_2}U,\cdots,
\partial_{x_N}U,\frac{N-2s}{2}U+x \cdot \nabla U\}$. The proof is completed.

\section*{Acknowledgments}

The authors are very grateful for the anonymous reviewers for their careful readings the manuscript and the valuable comments. 

This work was supported by National Natural Science Foundation of China (12261107) , Yunnan Fundamental Research Projects (202201AU070031, 202401AT070123), Scientific Research Fund of Yunnan Educational Commission (2023J0199) and Yunnan Key Laboratory of Modern Analytical Mathematics and Applications (202302AN360007).
\section*{Conflict of Interests}
The Author declares that there is no conflict of interest.
\bibliographystyle{plain}
\bibliography{yang}

\end{document}